\documentclass[12pt]{article}
\usepackage{amsbsy}
\usepackage{amssymb,amsfonts,amsthm,amsmath,bbm,mathrsfs,aliascnt,mathtools}
\usepackage{float,inputenc,upgreek}
\usepackage{doi,url}
\usepackage[numbers]{natbib}
\usepackage{etoolbox,letltxmacro,xifthen,ifdraft} % reinforce the environment
\usepackage[dvipsnames]{xcolor}
\usepackage{amsmath,amssymb,enumitem,bbm,mathrsfs,amsthm,aliascnt,mathtools}
\usepackage{graphicx,subcaption}
\usepackage{hyperref} % should be last
\hypersetup{final,hidelinks}
\usepackage[atend]{bookmark} % should be laster
\providecommand{\email}[1]{\href{mailto:#1}{\nolinkurl{#1}}}

\setlist[enumerate,1]{label={(\roman*)}}
\setlist[enumerate,2]{label={(\alph*)}}
\setlist[enumerate,3]{label={(\Roman*)}}
\newcommand{\newsstheorem}[2]{
  \newaliascnt{#1}{dummy}
  \newtheorem{#1}[#1]{#2}
  \aliascntresetthe{#1}
  \expandafter\def\csname #1autorefname\endcsname{#2}
}
\numberwithin{dummy}{section}
\theoremstyle{plain}
  \newsstheorem{theorem}{Theorem}
  \newsstheorem{proposition}{Proposition}
  \newsstheorem{corollary}{Corollary}
  \newsstheorem{lemma}{Lemma}
  \newsstheorem{definition}{Definition}
  \theoremstyle{definition}
  \newsstheorem{example}{Example}
  \newsstheorem{question}{Question}
\theoremstyle{remark}
  \newsstheorem{remark}{Remark}

\newenvironment{eqnarr*}{\begin{IEEEeqnarray*}{rCl}}{\end{IEEEeqnarray*}\ignorespacesafterend}

\newcommand{\N}{\mathbb{N}}
\newcommand{\Z}{\mathbb{Z}}
\newcommand{\R}{\mathbb{R}}

\newcommand{\calC}{\mathcal{C}}
\newcommand{\calCp}{\mathcal{C}_{>0}}

\renewcommand{\S}{\mathbf{S}}
\newcommand{\Pp}{\mathbf{P}}
\newcommand{\Ee}{\mathbf{E}}
\newcommand{\X}{\mathbf{X}}

\renewcommand{\bar}[1]{{#1}}
\renewcommand{\tilde}[1]{\widetilde{#1}}

\newcommand{\e}{\mathrm{e}}
\newcommand{\q}{\boldsymbol{\Theta}}

\DeclareMathOperator{\E}{\mathbb{E}}

\renewcommand{\P}{\mathbb{P}}

\newcommand{\calX}{\mathcal{X}}
\newcommand{\calP}{\mathcal{P}}
\newcommand{\calS}{\mathcal{S}}
\newcommand{\s}{\mathbf{s}}
\renewcommand{\a}{\mathbf{a}}
\newcommand{\g}{\mathbf{g}}
\newcommand{\G}{\mathbf{G}}
\newcommand{\kk}{\mathbf{k}}
\newcommand{\Kk}{\mathbf{K}}

\newcommand{\z}{\boldsymbol{\zeta}}
\renewcommand{\t}{\mathbf{t}}
\renewcommand{\d}{\mathrm{d}}
\newcommand{\eps}{\varepsilon}

\makeatletter \newcommand\mathof[1]{{\operator@font#1}} \makeatother
\begin{document}
\title{ Preferential relocations enhance survival \\
for  Markov chains with killing}

\author{
Jean Bertoin$^{*}$
\and
Martin Minchev$^{*,\dagger}$
}

\date{}
\maketitle
\thispagestyle{empty}

\begingroup
\renewcommand{\thefootnote}{\fnsymbol{footnote}}
\footnotetext[1]{Institute of Mathematics, University of Zurich, Switzerland.}
\footnotetext[2]{Faculty of Mathematics and Informatics, Sofia University, Bulgaria.}
\footnotetext[3]{The work of the second author was supported by the SNSF SCIEX programme, grant IZSF-0-235789.}
\endgroup

\renewcommand{\thefootnote}{\arabic{footnote}}
\setcounter{footnote}{0}
\thispagestyle{empty}

\begin{abstract} 
{We investigate the impact on survival of a modification of the evolution of a sub-stochastic Markov chain  that involves random relocations at previously visited states. 
Our central result is that such preferential relocations  increase the persistence rate, meaning the survival probability decays more slowly than for the  benchmark chain without relocations, and this improvement is strict under mild assumptions. We derive explicit lower bounds on the
ratio of persistence rates  when the relocation distribution is highly dispersed. The analysis relies on ergodic properties of so-called chains with complete connections, and a Feynman--Kac approach to estimate persistence. }
\newline  \vskip 1mm
{\normalfont \bfseries Keywords:}
{Markov chain with killing, persistence rate,  chain with complete connections, preferential relocation, spectral radius.}
\newline  \vskip 1mm
{\normalfont \bfseries Mathematics Subject Classification:}  60J05; 60J10; 60F10; 37A30; 37A50.

\end{abstract} 

\section{Introduction} \label{Sec:intro}

The notion of stochastic resetting has been first developed  in the physics literature as an intermittent search strategy such that  sequences of steps --called runs-- driven by a Markov kernel alternate with resetting events at which the process is repositioned at its initial state. We refer to the survey \cite{EMS} for  motivations and examples of applications,  explicit calculations, and an extensive list of references published before 2020. Preferential relocation refers to a variation of the former, such that after each  run,  the process is rather repositioned at a state visited
at some random time in its past, instead of at a fixed state; see \cite[Section 7]{EMS}. This has been proposed in \cite{BSS} as a  model for territorial exploration by an animal, and hence sometimes referred to as a monkey walk;  see \cite{MUB, BM1, BM2}. It also appears in models for dormancy, including various seed bank models in population dynamics, atavism in evolutionary biology, etc. See the survey \cite{BKSW} and references therein, and also \cite{BerBer}.

We are interested here in the following simple random evolution in this vein. Let $\sigma$ be an irreducible and aperiodic sub-stochastic matrix; it drives a Markov chain with killing which serves us as benchmark. We modify the dynamic at each step  by making a transition, still according to $\sigma$, but from the location of the chain $T$ units of time in the past rather than from its current location, where $T\geq 0$ is a random integer with a fixed law $\tau$, and independent copies of $T$ are used at all steps. In this setting, runs correspond to successions of steps for which $T=0$, and thus have a geometric distribution with success parameter given by the probability that $T\geq 1$, whereas events for which $T\geq 1$ correspond to preferential relocations. 

Let us sketch an interpretation  of one of  our main results --specifically Theorem \ref{T2}(ii)-- in the framework of territorial exploration. Think of a finite connected graph as a territory and suppose that some sites --possibly all--  are deleterious, that is, for each site, there is a risk to get killed when visiting that place. Consider  two monkeys, the first is named B for benchmark, the second R for relocation.   Suppose that the same sub-stochastic matrix governs the random displacements   of the two monkeys along the edges of the graph; more precisely those of B follow the ordinary killed Markov chain, whereas those of  R also incorporate preferential relocations based on some non-degenerate random variable $T$ with a finite expectation. Imagine now that we observe the monkeys after some large time; most likely both are already  dead. However, conditionally on that one of the two survived, it is much more probable that the survivor is R than B. 
More formally, the central observation of this work is that the persistence rate, which measures the decay of the survival probability over time, is consistently larger for the killed chain with relocations than for the benchmark killed Markov chain without them\footnote{Hence the title that preferential relocations enhance survival.
Let us briefly comment on this claim to avoid a possible misinterpretation. Stochastic resetting typically accelerates the search of a target, in the sense that the average time for the process to find  the target is shorter when resetting is incorporated to the dynamics. The hitting time can be seen as  the lifetime for the search process, since its goal has been accomplished. This does not contradict our results which have a different focus and concern the tail distribution of the lifetime rather than its expectation. }. A weaker version of this observation was established in \cite{BerBer} within the context of a multitype branching process with memory, and the primary objective of the present work is to provide a deeper understanding of this phenomenon. 

Our analysis relies on the fact that dynamics with preferential relocations fall under the scope of chains with complete connections, introduced by Doeblin and Fortet \cite{DF} and extensively studied in ergodic theory. In short, we consider the substochastic version of these processes and apply a Feynman--Kac representation to derive general lower bounds for their persistence rates. An important aspect of this approach is the freedom in choosing the specific conservative chain with complete connections to which the Feynman--Kac formula is applied, similar to the flexibility found in the Donsker--Varadhan large deviation theory for Markov chains. 
 A first qualitative new result in this direction is that, under mild assumptions on the law $\tau$ and/or  the matrix $\sigma$, the inequality for the persistence rates is always strict, except in the case when $\tau$ is a Dirac mass (i.e. except when the random variable $T$ in the preceding paragraph is degenerate). A second one is that this inequality persists even when we restrict the empirical measure of the chain to be in a fixed open set. Furthermore, we establish quantitative results, notably deriving an explicit asymptotic lower bound for the spectral radius of the killed chain with relocations when $\tau$ is highly dispersed. Numerical simulations suggest this asymptotic bound may be sharp.

The remainder of this work is organized as follows. Section \ref{Sec:chainmemo} concerns chains with complete connections, which encompass more general processes than the preferential relocation models of interest here. We recall known criteria for ergodicity in the conservative case before turning to the defective case. Our principal result in this setting is a general lower bound for the persistence rate, stated in Theorem \ref{T1}. In Corollary \ref{C1}, we derive a simple formula for the spectral radius in terms of an optimization problem, which shares similarities with well-known variational formulas in the literature. 
Killed chains with preferential relocations are formally introduced in Section \ref{sec:persisbound}, where some features presented in Section \ref{Sec:chainmemo} are specialized. Section \ref{sec:PF} is devoted to two applications of the Perron--Frobenius theorem. Theorem \ref{T2} states rigorously that preferential relocations enhance persistence, formalizing the discussion above. Theorem \ref{T3} presents a large deviation lower bound for the empirical measures of the killed chain, echoing the celebrated Donsker--Varadhan result. It shows that as the number of steps goes to infinity, the probability for the empirical measure to be found in a given open set is always higher when preferential relocations are incorporated into the evolution. Finally, Section \ref{sec:disperse} addresses asymptotics when the relocation law $\tau$ is widely dispersed. Theorem \ref{T4} asserts a general concentration result for occupation measures of conservative chains, which immediately yields the asymptotic lower bound for the spectral radius stated in Corollary \ref{C3}. Intriguingly, numerical simulations suggest this asymptotic lower bound might actually be a general upper bound valid for all laws $\tau$, though we have yet to find an explanation.

\section{Chains with complete connections and killing} \label{Sec:chainmemo}
\subsection{Ergodicity of chains with complete connections}
The purpose of this section is to recall some important  ergodic properties of certain  Markovian dynamics, known as chains with complete connections,  which will be useful for  our analysis. The  terminology is borrowed from  Doeblin and Fortet \cite{DF}; other names, such as chains with unbounded memory, or with infinite order, or Doeblin chains, are also commonly used in the literature. 

Let $S$ be some finite set; we write $\S=S^{\N}$ for the space of infinite sequences in $S$ and systematically use bold letters to designate objects (functionals, measures, etc.) related to  $\S$. Notably $\s=(s_0, s_1, \ldots)$ denotes a generic element of $\S$, and we then simply write $t\s =(t,s_0, s_1, \ldots)$ for the sequence obtained by adding $t\in S$ as prefix to $\s$.
We endow $\S$ with the distance 
$$\mathbf{d}(\s,\mathbf{t})= 2^{-j} \Longleftrightarrow s_i=t_i \text{ for all }i<j \text{ and } s_j\neq t_j  ,$$
so that $(\S,\mathbf{d})$ is a compact metric space and its topology is that of pointwise convergence.

Next consider a continuous functional $\g: \S\to [0,1]$ with
$$\sum_{t\in S} \g(t\s)=1,\qquad \text{ for all }\s\in \S.$$
This induces a continuous kernel $\s \mapsto \g(\cdot\s)$ from $\S$  to the space $\calP(S)$ of probability measures  on $S$ equipped with the topology of weak convergence,  and  a transition (or transfer) operator $\G$ on the space $\calC(\S)$ of continuous functionals on $\S$ equipped with the supremum norm,
such that  for every $\mathbf{f}\in \calC(\S)$, 
$$ \G\mathbf{f}(\s)=\sum_{t\in S}\g(t\s)\mathbf{f}(t\s).$$
The Feller property holds, that is, $\G\mathbf{f}\in \calC(\S)$ for every  $\mathbf{f}\in \calC(\S)$. Plainly, $\G$ maps the space of positive continuous functionals 
$$\calCp(\S)=\{\mathbf{f}\in \calC(\S): \mathbf{f}(\s)>0\text{ for all }\s\in \S\}$$
 into itself and we call $\G$ positive.

We denote by $\Pp_{\boldsymbol{\nu}}$ the law of the Markov chain $\X=(\X(n))_{n\geq 0}$ on $\S$  with transition operator $\G$ and  initial distribution  $\boldsymbol{\nu}\in\calP(\S)$. We think of $\X$ as a growing  sequence of backward paths in $S$; more precisely if $ \X(n)=(X_0(n), X_1(n), \ldots)$, then  $X_j(n)=X_{i+j}(n+i)$  for any $i,j\geq 0$. The chain keeps full memory of its past, in the sense that for any $n\leq n'$, $\X(n)$ can be recovered from $\X(n')$ as a suffix of the latter;  as a consequence the strong Feller property fails and $\X$ is not a Harris chain.  Analytically, we have that for every $n\geq 0$ and $\mathbf{f}\in \calC(\S)$
$$\Ee_{\boldsymbol{\nu}}\left( \mathbf{f}(\X(n))\right) = \int_{\S} {\boldsymbol{\nu}} (\d \s)\G^n\mathbf{f}(\s) =  \int_{\S} \G^{*n}{\boldsymbol{\nu}} (\d \s)\mathbf{f}(\s),$$
where  $\G^*$ denotes the dual transition operator acting on the space of signed measures on $\S$.

 It is a well-known consequence of the Schauder--Tychonoff fixed point theorem that any Feller chain on a compact space possesses an invariant distribution. In the present setting, there exists $\boldsymbol{\mu} \in \calP(\S)$, which we call, following \cite{BCJO},   a Doeblin measure, such that $\G^*\boldsymbol{\mu}  = \boldsymbol{\mu} $ . One says
 that the kernel $\g$ is uniquely ergodic if the Doeblin measure is unique. One further says  that $\g$ satisfies  dynamic uniqueness if there exists  $\boldsymbol{\mu}\in \calP(\S)$ such that,
for every distribution  $\boldsymbol{\nu}\in\calP(\S)$, 
$$\lim_{n\to \infty} \G^{*n}{\boldsymbol{\nu}} = \boldsymbol{\mu}, \qquad\text{weakly on }\calP(\S).$$
In that case, $\boldsymbol{\mu}$ is a Doeblin measure  and  $\g$ is uniquely ergodic. We refer to Definition 3 and Theorem 1 in \cite{GGT} for equivalent formulations of dynamic uniqueness, notably in terms of coupling, mixing, and triviality of a tail sigma-algebra. 

We now lift from the literature two handy criteria ensuring unique ergodicity; they are both expressed in terms of the variation of the kernel,
$$\mathrm{var}_j(\g)= \sup\{| \g(\s)-\g(\t)|: \mathbf{d}(\s,\t)\leq 2^{-j}\}, \qquad \text{for } j\geq 0.$$
The positivity assumption of the kernel in Theorem \ref{L:ergJOP} below is stronger than topological mixing
in Theorem \ref{L:ergDF}, but the condition on the variation of the kernel is weaker. 
The first basic criterion, tailored for our purpose,  is implicit in \cite{DF}; see also \cite{Keane} and \cite{Walters}.

\begin{theorem}[Doeblin and Fortet] \label{L:ergDF} Assume the kernel $\g$ is  topological mixing, in the sense that for every $\eps>0$, there exists $n=n(\eps)$ such that for 
any two states $\s$ and $\t$ in $\S$,  $$\Pp_{\s}( \mathbf{d}(\X(n),\t)<\eps)>0.$$
If furthermore the kernel $\g$ has summable variation, i.e. 
$$\sum_{n\geq 0}\mathrm{var}_n(\g)<\infty,$$
then dynamical uniqueness, and a fortiori unique ergodicity, hold.
\end{theorem}

The second criterion is taken from \cite{JOP}.  
\begin{theorem}[Johansson, \"Oberg, and Pollicott]  \label{L:ergJOP}  Suppose that the continuous kernel  $\g$ is positive, i.e. $\g(\s)>0$ for all $\s\in \S$, and satisfies
$$\mathrm{var}_n(\g)=o(n^{-1/2}), \qquad \text{as }n\to \infty.$$
Then  unique ergodicity holds.
\end{theorem}

We mention that  \cite{BCJO} recently established a weaker sufficient condition for unique ergodicity. We also refer  to \cite{BK, BHS} for cases where unique ergodicity fails.
Further,  \cite{JO1}  proved first that the stronger assumption of  square summability of the variation of the kernel ensures unique ergodicity, and then
\cite{JO2} and \cite{GGT} that this stronger  condition also ensures dynamic uniqueness.

\subsection{Persistence bounds for killed chains with complete connections } \label{Sec:kcomplconn}
Killed chains with complete connections  arise for defective  (i.e. sub-stochastic) kernels. Specifically, we 
 consider here a continuous functional $\kk: \S\to [0,1]$ with 
$$\sum_{t\in S} \kk(t\s)\leq 1,\qquad \text{ for all }\s\in \S.$$
Just as in the preceding section, $\kk$ yields a positive Feller transition operator $\Kk$ on $\calC(\S)$,
where  for every $\mathbf{f}\in \calC(\S)$, 
$$ \Kk\mathbf{f}(\s)=\sum_{t\in S}\kk(t\s)\mathbf{f}(t\s).$$

The transition operator $\Kk$ is now only sub-Markovian; it  induces a defective Markov chain on $\S$, which we refer to as a killed chain with complete connections. 
Its asymptotic behavior is depicted by an extension of the Perron--Frobenius theorem  due to Ruelle \cite{Ruelle}, whenever the defective kernel $\kk$ is positive and H\"older-continuous. See \cite[Theorem 1]{FJ} or \cite[Theorem 8]{Walters2} for versions that encompass the present setting.

Write $\mathcal{H}_{>0}(\S)$ for the space of H\"older-continuous functionals $\a: \S\to (0,\infty)$,  that is,
$$|\a(\s)-\a(\t)| \leq c\, \mathbf{d}(\s, \t)^{\alpha},\qquad \text{for all }\s,\t\in \S,$$
for some $c<\infty$ and $\alpha>0$. 
Recall that the spectral radius $\mathbf{r}$ of the transition operator $\Kk$ is given   in terms of the operator norm $\|\cdot\|$ on the Banach algebra $\calC(\S)$  by the Gelfand formula, 
$$\mathbf{r} = \lim_{n\to\infty} \| \Kk^n\|^{1/n}.$$

\begin{theorem}[Ruelle] \label{L:Ruelle} Suppose that the defective kernel $\kk\in \mathcal{H}_{>0}(\S)$. Then there exist a unique probability measure $\boldsymbol{\mu}\in \calP(\S)$ --often called the quasi-stationary distribution-- and a unique functional
$\mathbf{h}\in \mathcal{H}_{>0}(\S)$  such that
\[\Kk \mathbf{h}= \mathbf{r}  \mathbf{h}, \quad \Kk^{*} \boldsymbol{\mu}= \mathbf{r}  \boldsymbol{\mu}, \quad \text{and} \quad \int_{\S}\boldsymbol{\mu}(\d \s) \mathbf{h}(\s)=1.
\]
Moreover, for any $\mathbf{f}\in \calC(\S)$, there is the convergence in $\calC(\S)$
\[ \lim_{n\to \infty} \mathbf{r}^{-n}\Kk^n \mathbf{f} =   \boldsymbol{\mu}( \mathbf{f}) \mathbf{h} , \qquad \text{where} \  \boldsymbol{\mu}( \mathbf{f})=\int_{\S}\boldsymbol{\mu}(\d \s) \mathbf{f}(\s).
\]
\end{theorem}

We denote the law of the killed chain with complete connections started with initial distribution $\boldsymbol{\nu}\in \calP(\S)$ by   $\P_{\boldsymbol{\nu}}$ (or simply $\P_{\s}$ when
$\boldsymbol{\nu}$ is the Dirac mass at $\s$), 
and its lifetime by $\z$.
We are interested in the asymptotic behavior  of the survival probability, also called persistence,  $\P_{\boldsymbol{\nu}}(\z>n)$ as $n\to \infty$.
Plainly, one has always
\[\limsup_{n\to \infty} \frac{1}{n} \log \P_{\boldsymbol{\nu}}(\z>n) \leq \log \mathbf{r}, \]
and if the assumptions in Theorem \ref{L:Ruelle} hold, then
\[\P_{\s}(\z>n)=\Kk^n\mathbf{1}(\s) \sim \mathbf{r}^{n}\mathbf{h}(\s).\]
In words, the persistence rate may depend on the initial distribution, it is always dominated by the spectral radius, and may also coincide the latter, notably when the conclusions of the Perron--Frobenius theorem hold.
We now 
state the key result of this section.

\begin{theorem} \label{T1}  Consider an arbitrary $\a\in \calCp(\S)$ such that
$\Kk\a(\s)>0$ for all $\s\in\S$.
Introduce the conservative biased kernel $\g$ given for any $t\in S$ and $\s\in \S$ by 
\begin{equation} \label{E.defgk}
\g(t\s) = \frac{ \kk(t\s) \a(t\s)}{\Kk \a(\s)}.
\end{equation}
Suppose that $\g$ is uniquely ergodic and write  $ \boldsymbol{\mu}_{\a}$ for its  Doeblin measure. Then for 
any
probability measure $\boldsymbol{\nu}$ on $\S$,
 there is the lower bound
 $$\liminf_{n\to \infty} \frac{1}{n} \log \P_{\boldsymbol{\nu}}(\z>n) \geq \int_{\S} \boldsymbol{\mu}_{\a}(\d \s) \log\left( \frac{ \Kk \a(\s)}{{\a}(\s)}\right) .
$$
 \end{theorem}

\begin{proof} 
We stress from the positive Feller property that  \eqref{E.defgk} defines a continuous kernel.
Next, observe that for every $\mathbf{f}\in \calC(\S)$, there is  the identity
$$\Kk\mathbf{f} = \G\mathbf{\tilde f}, \qquad\text{for } \mathbf{\tilde f}(t\s)= \mathbf{f}(t\s) \frac{\Kk \a(\s)}{{\a}(t\s)}.$$
By iteration, we get the Feynman--Kac formula 
$$
 \int_{\S} \boldsymbol{\nu}(\d \s) \Kk^n\mathbf{f}(\s) =\Ee_{\boldsymbol{\nu}}\left( \mathbf{f}(\X(n)) \prod_{j=1}^{n} \frac{\Kk \a(\X(j-1))}{{\a}(\X(j))}\right),
$$
where $\X=(\X(n))_{n\geq 0}$ is the chain with complete connections induced by the biased kernel $\g$.
We write $\mathbf{1}$ for the function identical to $1$ on $\S$, so that 
\begin{equation} \label{E:FK}
 \P_{\boldsymbol{\nu}}(\z>n) =   \int_{\S} \boldsymbol{\nu}(\d \s) \Kk^n \mathbf 1 (\s) = \Ee_{\boldsymbol{\nu}}\left( \prod_{j=1}^{n} \frac{\Kk \a(\X(j-1))}{{\a}(\X(j))}\right).
 \end{equation}
 Then by Jensen's inequality, 
  $$\liminf_{n\to \infty} \frac{1}{n} \log  \P_{\boldsymbol{\nu}}(\z>n) \geq \liminf_{n\to \infty}  \Ee_{\boldsymbol{\nu}}\left( \frac{1}{n} \sum_{j=1}^{n} \log\left(\frac{\Kk \a(\X(j-1))}{{\a}(\X(j))}\right)\right).
  $$
 
The functionals $\log \Kk \a$ and $\log \a$ are continuous on the compact metric space $\S$, and by an application of Breiman's strong law of large numbers for Feller Markov chains \cite{Breiman}, unique ergodicity ensures the $\Pp_{\boldsymbol{\nu}}$-almost-sure convergence
$$\lim_{n\to \infty}   \frac{1}{n} \sum_{j=1}^{n} \log\left(\frac{\Kk \a(\X(j-1))}{{\a}(\X(j))}\right) = \int_{\S} \boldsymbol{\mu}_{\a}(\d \s) \log\left( \frac{ \Kk \a(\s)}{{\a}(\s)}\right).$$
We finish the proof with Fatou's lemma, or Lebesgue's dominated convergence. 
\end{proof}

We next point at a simple consequence of Theorem \ref{T1} for the spectral radius $\mathbf{r}$ of  $\Kk$ for a H\"older continuous kernel.

\begin{corollary} \label{C1}
Suppose that the defective kernel $\kk$  belongs to $\mathcal{H}_{>0}(\S)$. Then for any $\a\in \mathcal{H}_{>0}(\S)$,  the conservative biased  kernel $\g$ defined in \eqref{E.defgk} is uniquely ergodic. We write $\boldsymbol{\mu}_{\a}$ for its Doeblin measure and 
have the identity
\[
\log \mathbf{r} = \sup_{\a\in \mathcal{H}_{>0}(\S)} \int_{\S} \boldsymbol{\mu}_{\a}(\d \s) \log\left( \frac{ \Kk \a(\s)}{\a(\s)}\right).
\]
\end{corollary}

\begin{proof} 
On the one hand,  the Gelfand formula can be formulated in terms of persistence rates as
$$\log \mathbf{r} = \lim_{n\to\infty} \sup_{{\boldsymbol{\nu}\in \calP(\S)}} \frac{1}{n} \log \P_{\boldsymbol{\nu}}(\z>n) \geq  \sup_{{\boldsymbol{\nu}\in \calP(\S)}} \liminf_{n\to\infty} \frac{1}{n} \log \P_{\boldsymbol{\nu}}(\z>n).$$
It readily follows from our assumptions that $\g$ fulfills the requirement of Theorem \ref{L:ergJOP} and the biased chain   is
 uniquely ergodic. Thus 
Theorem \ref{T1} yields the inequality
\begin{equation} \label{E:specradrel}
\log \mathbf{r} \geq \sup_{\a\in \mathcal{H}_{>0}(\S)} \int_{\S} \boldsymbol{\mu}_{\a}(\d \s) \log\left( \frac{ \Kk \a(\s)}{\a(\s)}\right).
\end{equation}

On the other hand, we invoke Ruelle's Theorem  \ref{L:Ruelle}. Recall that the eigenfunction $\mathbf{h}$ there is H\"older-continuous, so we can take $\mathbf{a}=\mathbf{h}$ in the right-hand side of \eqref{E:specradrel} and the converse inequality follows.
\end{proof}

Corollary \ref{C1} bears some resemblance with the Donsker--Varadhan--Friedland variational formula for the spectral radius, which reads in our setting 
\[
\log \mathbf{r} = \sup_{\boldsymbol{\nu}\in \calP(\S)} \inf_{\a\in \calCp(\S)} \int_{\S} \boldsymbol{\nu}(\d \s) \log\left( \frac{ \Kk \a(\s)}{\a(\s)}\right).
\]
We refer to \cite{Fried} for further variational formulas for positive operators that fulfill the conclusions of the Perron--Frobenius theorem.

\section{Persistence bounds for  killed chains with relocations} \label{sec:persisbound}
Killed chains with preferential relocations  are the main objects of interest in the rest of this work. They were described informally in the Introduction, and we now introduce them formally as a natural two-parameter family of chains with complete connections, where the first parameter is a sub-stochastic matrix $\sigma$ driving ordinary transitions, and the second is a law $\tau$ on $\N$ governing preferential relocations.  We are mainly interested in their persistence rates and will rephrase in this setting the general lower bound obtained in Theorem \ref{T1}. This requires unique ergodicity for  chains with weighted relocations that arise in this setting.

Let $\sigma=(\sigma(s,t))_{s,t\in S}$ be a sub-stochastic matrix, i.e. the entries are nonnegative  with 
$$\sum_{t\in S} \sigma(s,t)\leq 1,\qquad \text{ for every }s\in S;$$ 
we further assume $\sigma$ to be irreducible and aperiodic. So $\sigma$ drives a Markov chain with killing, where for every $s\in S$, the probability that the chain located at $s$ being killed at its next step equals $1 - \sum_{t\in S} \sigma(s,t)$.
We always suppose that $\sigma$ is not proportional to a stochastic matrix, to ensure that the killing probabilities are not the same at all states $s$, as otherwise the questions we are interested in would be trivial. 

We modify the evolution of the sub-Markov chain and make it depend on its past using further a probability law $\tau$ on $\N$, which we call the relocation law.
That is, in the framework of Section \ref{Sec:kcomplconn}, we consider the kernel $\kk$ on $\S$ given by 
$$ \kk (t\s)=  \sum_{i=0}^{\infty} \tau(i)  \sigma(s_i,t), \qquad \text{for }\s\in \S \text{ and }t\in S.$$
In words, the dynamics are that of a walk with memory in $S$ such that at each step, with probability $\tau(0)$, the transition is made according to $\sigma$, and with complementary probability $1-\tau(0)$, we rather relocate the walk at its state $R$ units of time before and then make instantaneously another transition  from that state according to $\sigma$, where $R\geq 1$ is random with distribution $
\left(\tau(j)/(1-\tau(0))\right)_{j\geq 1}$.  Relocations are preferential, in the sense that they depend on the past trajectory, as opposed to different deterministic resetting where the walk is rather sent back to a fixed location. 
We thus call  the killed chain with complete connections in Section \ref{Sec:kcomplconn} for the defective kernel $\kk$, the killed chain with relocations (for the sub-stochastic matrix $\sigma$ and the law $\tau$). 

Recall that the law of  the killed chain with relocations with initial law $\boldsymbol{\nu}\in \calP(\S)$ is denoted by $\P_{\boldsymbol{\nu}}$,  and that
$\z$ stands for its lifetime.
We are chiefly interested in the asymptotic behavior as $n\to \infty$ of the survival probability $\P_{\boldsymbol{\nu}}(\z>n)$
 and  want to compare it to the benchmark model without relocations. This requires the introduction of (conservative)
 chains with weighted relocations.

Fix some arbitrary positive function $a\in (0,\infty)^S$ that we view as a column vector  and define the stochastic matrix $\pi=(\pi(s,t))_{s,t\in S}$ with biased entries 
$$\pi(s,t)= \frac{{a}(t)}{\sigma a(s)}\sigma(s,t)  .$$
Beware that $\pi$ depends on the choice of the function ${a}$,  even though for simplicity  this does not appear in the notation.
Next recall \eqref{E.defgk}, which can be rewritten in the present setting for a functional  of zero order $\a(\s)=a(s_0)$ as 
$$ \Kk a(\s)=\sum_{i=0} ^{\infty}\tau(i) \sigma a(s_i)=\sum_{i=0}^{\infty}\sum_{t\in S} \tau(i) \sigma(s_i,t){a}(t), \qquad \text{for }\s\in \S,$$
and then the biased kernel  $\g$ is
\begin{equation} \label{E:boldpi}
\g(t\s)=\frac{1}{\Kk a(\s)}  \sum_{i=0}^{\infty} \tau(i) \sigma(s_i,t) {a}(t) =\frac{1}{\Kk a(\s)}  \sum_{i=0}^{\infty} \tau(i) \sigma a(s_i) \pi(s_i,t)   .
\end{equation}
We call 
the chain with complete connections  and kernel  $\g$, the  chain with weighted relocations (with parameters the sub-stochastic matrix $\sigma$, the law $\tau$, and the function $a$). 
The adjective weighted refers to the second expression in  \eqref{E:boldpi}, which can be interpreted as incorporating relocations to the (conservative)   Markovian dynamics driven by $\pi$, where now relocations depend not only on the law $\tau$, but also on  weights induced by $\sigma a$.

We aim at applying the general lower bound for the persistence rate of Theorem \ref{T1} and have to check that the assumptions there are fulfilled in the present setting.
That $\Kk a(\s)>0$  for all $\s$ is immediate by irreducibility of $\sigma$.   The continuity of $\kk$ is plain; beware however that positivity may fail when some entries of $\sigma$ are null.
This leaves us with the issue of the unique ergodicity of $\g$, which motivates the following.  

We first  make a couple of simple observations about  chains with weighted relocations; recall that they are denoted generically by $\X=(\X(n))_{n\geq 0}$ and that we write $\Pp$ for their distributions. We first consider the elementary case when the law $\tau$ has bounded support,  say with maximal element $d$, i.e.
$$\tau(d)>0=\tau(d+j),\quad \text{ for all }j\geq 1. $$ 
We then  write $\calS=S^{\{0, \ldots, d\}}$ and $\calX$ for the image of $\X$ by the projection from $\S$ to $\calS$; plainly $\calX$ is a Markov chain.   
\begin{lemma}\label{L:taubounded} Suppose that the law $\tau$ has bounded support. Then the Markov chain $\calX$ on the finite space $\calS$  is 
irreducible and aperiodic.\end{lemma}
\begin{proof} Consider two arbitrary states in $\calS$, $(s_0, \ldots, s_d)$ and $(t_0, \ldots, t_d)$. We work with the Markov chain $\calX$ started from the first and have to
check that after some deterministic number of steps, the probability that it is found at the second state is strictly positive. 

The biased matrix $\pi$ inherits  irreducibility and aperiodicity from the sub-stochastic matrix $\sigma$, which
means that there is some large enough $k\geq 1$ such that all the entries of $\pi^k$ are strictly positive.  For every $j=0, \ldots, d$, we can thus find a sequence with length $k+1$
in $S$, say $u_j(\ell)$ for $\ell=0, \ldots, k$, starting from $u_j(0)=s_j$ and finishing at $u_j(k)=t_j$, such that $\pi(u_j(\ell-1), u_j(\ell))>0$ for every $\ell=1, \ldots, k$.

Since $\tau(d)>0$, the probability that after one step, 
$\calX(1)=(u_d(1),u_0(0), \ldots, u_{d-1}(0))$, is strictly positive.
By iteration, the probability after $d$ more steps we have
$\calX(d+1)=(u_0(1), \ldots, u_d(1))$,  is also strictly positive. 
We conclude from the Markov property that the probability that
$\calX(k(d+1))=\t$ is strictly positive.
\end{proof}
We stress that aperiodicity of $\pi$ is crucial: for instance, it is easy to see that if $\pi$ is irreducible and $2$-periodic, and if $\tau$ has support in $2\N$, then the chain with weighted relocations $\calX$ is reducible. 

We now turn our attention to the case when $\tau$ has unbounded support and establish a weaker version of Lemma \ref{L:taubounded}.
Recall the notion of topological mixing in Theorem \ref{L:ergDF}.

\begin{lemma}\label{L:tauunbounded}  For any  $a\in (0,\infty)^S$, topological mixing  holds for  the chain with weighted relocations and kernel $\g$ given in \eqref{E:boldpi}, and any Doeblin measure of the latter  has full topological support.
\end{lemma}

\begin{proof}By definition of the distance $\mathbf{d}$ on  $\S$ and an argument similar to that for Lemma \ref{L:taubounded}, we can find for any $\ell\geq 1$ an integer $k(\ell)$ sufficiently large such that
\begin{equation}\label{E:strongi}
\inf_{\s,\t \in \S} \Pp_{\s}\left( \mathbf{d}(\X(k(\ell)),\t)\leq 2^{-\ell}\right )>0.
\end{equation}
The assertion that any invariant law for $\G^*$ has full topological support stems from \cite[Proposition 5.8(ii)]{BH}.
\end{proof}

We stress that topological mixing is a weaker property than irreducibility, and that contrary to the latter, it does not ensure uniqueness of the Doeblin measure (see \cite[Theorem 5.5]{BH}).

\begin{proposition}\label{P:chainmem}  Unique ergodicity holds for the kernel $\g$ given in \eqref{E:boldpi}  
for an arbitrary  $a\in (0,\infty)^S$ whenever either of the following two conditions are fulfilled:
\begin{enumerate}
\item[(i)] The law $\tau$ has a finite first moment, i.e.
\begin{equation} \label{E:tauf1}
\sum_{j\geq 1} j \tau(j) <\infty.
\end{equation}

\item[(ii)]  The sub-stochastic matrix $\sigma$ is positive, i.e. $\sigma(s,t)>0$ for all $s,t\in S$, and the tail of $\tau$ satisfies 
\begin{equation} \label{E:tauf2}
\sum_{i\geq n} \tau(i) =o(n^{-1/2}),\qquad \text{ as }n\to \infty.
\end{equation}
\end{enumerate}
Last, the defective kernel $\kk$ is H\"older-continuous whenever the tail of the law $\tau$ decays exponentially, i.e.
$$ \sum_{j\geq n} \tau(j) = o(b^{-n}), \qquad \text{for some }b>1.$$

\end{proposition}

\begin{proof} Plainly, by \eqref{E:boldpi},  the variations of $\g$ can be controlled by the tail of $\tau$. Specifically
 for any $j\geq 0$ and any $\s,\s'\in \S$
with $\mathbf{d}(\s,\s')\leq 2^{-j}$ (recall that this means that the prefixes of length $j$ of $\s$ and $\s'$ are identical), there is the bound
$$|\g(\s)-\g(\s')| \leq c(\sigma, a)\frac{\max \sigma a}{\min \sigma a} \sum_{i\geq j} \tau(i),$$
where $c(\sigma, a)>0$ is some constant depending on $\sigma$ and $a$ only.
Since topological mixing is granted from Lemma \ref{L:tauunbounded}, (i) follows from Theorem \ref{L:ergDF}.
In turn, when $\sigma$ has positive entries, $\g$ is positive and  (ii) follows from Theorem \ref{L:ergJOP}.
The last claim should also be clear.
\end{proof} 

We can now state the main result of this section, which specializes   
Theorem \ref{T1} to the present setting.
\begin{corollary} \label{C2} 
Suppose either that \eqref{E:tauf1} holds, or that $\sigma(s,t)$ has positive entries
  and \eqref{E:tauf2} holds. 
Then for any
probability measure $\boldsymbol{\nu}$ on $\S$,
 there is the lower bound
 $$\liminf_{n\to \infty} \frac{1}{n} \log \P_{\boldsymbol{\nu}}(\z>n) \geq \sup_{a\in(0,\infty)^S}  \int_{\S} \boldsymbol{\mu}_a(\d \s) \log\left( \frac{ \Kk a(\s)}{{a}(s_0)}\right),
$$
where $ \boldsymbol{\mu}_a$ denotes the unique Doeblin measure of the kernel $\g$ in \eqref{E:boldpi}.
 \end{corollary}

The bound of Corollary \ref{C2} is in general weaker than that in Theorem \ref{T1}, because it only involves functionals $\a(\s)=a(s_0)$ that only depend on the zeroth term of $\s$. It  may still look rather obscure at first sight, as it is expressed in terms of  the Doeblin measures $\boldsymbol{\mu}_a$ which are not known explicitly.  We shall see in the next section that much more specific information can nonetheless be deduced.

We now conclude this section by pointing at an alternative expression for the lower bound in Corollary \ref{C2}
 in terms of the kernel  $\q: \S\to \calP(S)$, where for every $\s\in \S$,  $\q(\s)$ is 
 the row vector $(\q(\s,t))_{t\in S}$
with coordinates
\begin{equation} \label{E:occmeas}
        \q(\s,t)
        =
        \sum_{i: s_i=t}
        \tau(i), \qquad \text{for }t\in S.
      \end{equation}
      In words, $\q(\s)$ is the occupation measure of $\s$ on $S$ relative to the law $\tau$.
Writing any function on $S$ as a column vector, we have the expression
\[
        \Kk \bar a(\s)
        =
        \q(\s)\sigma a.
        \]
On the other hand, since all one-dimensional marginals of $\boldsymbol{\mu}_a$ are
identical, there is the identity
\[
        \int_{\S}\boldsymbol{\mu}_a(\d\s)\log a(s_0)
        =
        \int_{\S}\boldsymbol{\mu}_a(\d\s) \q(\s) \log a
      .
\]
We conclude that
\begin{equation} \label{E:T1alt}
        \int_{\S}
        \boldsymbol{\mu}_a(\d\s)
        \log\left(
        \frac{\Kk \bar a(\s)}{a(s_0)}
        \right)
        =
        \int_{\calP(S)}
        \left(\boldsymbol{\mu}_a\circ \q^{-1}\right)(\d p)
        \Big(
        \log(p \sigma a)
        -
       p \log a
        \Big),
\end{equation}
where $\boldsymbol{\mu}_a\circ \q^{-1}$ denotes  the pushforward measure of $\boldsymbol{\mu}_a$ by $\q$.

\section{Two consequences of the Perron--Frobenius theorem} \label{sec:PF}

Recall from the Perron--Frobenius theorem that  there is the convergence
 \begin{equation}\label{E:PF}\lim_{n\to \infty} r^{-n} \sigma^n(s,t) = h(s)\uprho (t), \qquad
  \text{for }s,t\in S,
 \end{equation}
 where $r\in (0,1)$ denotes the spectral radius of $\sigma$ (often also called the Perron--Frobenius eigenvalue), $\uprho \in \calP(S)$ is a probability measure on $S$ with full support,  and $h:S\to (0,\infty)$ is a positive function on $S$ 
 such that 
 \begin{equation} \label{E.specm}
\uprho \sigma  = r \uprho, \quad \sigma h=r h, \quad \text{and } \uprho h=\sum_{s\in S} \uprho (s) h(s) =1.
\end{equation}
Moreover these properties determine the spectral elements, in the sense that if $r'$ is another real number and $h'$ a nonnegative function on $S$ with $h'\not \equiv 0$
such that $\sigma h'=r'h'$, then $r'=r$ and $h'=ch$ for some $c>0$.
We stress that in  \eqref{E:PF} and  often later on, we  identify functions on $S$ as column vectors and probability measures on $S$ as row vectors, and apply the usual conventions for matrix multiplication. The probabilistic interpretation of \eqref{E:PF}  identifies the Perron--Frobenius eigenvalue
 $r$ as the persistence rate and  $\uprho$ is often referred to as the quasi-stationary distribution.

The observation at the heart of this section is that, if we choose  $a=h$, then  $\sigma h=rh= r a$. 
This will greatly simplify the lower bound in Corollary \ref{C2}. On the other hand, we stress that despite its simplicity, this choice is not optimal and different functions may yield sharper lower bounds.

 \subsection{Relocations  enhance persistence}
As a first application, we now show that the persistence rate for the killed chain with relocations, started with an arbitrary memory,
 is always at least as large as that for the benchmark. Here is the formal statement.

 \begin{theorem} \label{T2}
  For any probability measure $\boldsymbol{\nu}$ on $\S$, the following assertions hold:
  \begin{enumerate}
   \item[(i)] We always have
   $$\liminf_{n\to \infty} \frac{1}{n} \log  \P_{\boldsymbol{\nu}}(\z>n) \geq \log r. $$

  \item[(ii)] Suppose that $\tau$ is not a Dirac point mass and furthermore, either that \eqref{E:tauf1} holds, or that all the entries $\sigma(s,t)$ are positive 
  and \eqref{E:tauf2} holds. Then there is the strict inequality
   $$\liminf_{n\to \infty} \frac{1}{n} \log  \P_{\boldsymbol{\nu}}(\z>n)> \log r. $$

  \item [(iii)] If  $\tau$ is  a Dirac point mass, then 
$$ \lim_{n\to \infty} r^{-n} \P_{\boldsymbol{\nu}}(\z>n) = c(\boldsymbol{\nu})$$
for some $0<c(\boldsymbol{\nu})<\infty$. 

  \end{enumerate} 
  
   \end{theorem}
 
 Theorem \ref{T2} extends the main result in \cite{BerBer} which is expressed in terms of  spectral radii; specifically (i) is proven there for a special choice of the probability measure $\boldsymbol{\nu}$ only and the cases where the inequality is actually an equality were not identified. The approach in \cite{BerBer}  relies on similar ideas, with some technical differences however (notably the chains involved there rather  lived in the product space $S\times  \S$
 and unique ergodicity is not needed there).

\begin{proof} It is convenient to prove the claims in their reversed order,  from the most specific to the most general.

(iii) Suppose that $\tau$ is a Dirac mass, say at $d\geq 0$. Then the dynamics of the killed chain with relocations  can be realized by intertwining $d+1$ independent killed Markov chains on $S$ driven by  $\sigma$, and the claim is thus immediate from the Perron--Frobenius theorem. 

(ii) We take here $a=h$, so  $\sigma a=rh$, and also recall that,  thanks to  Proposition \ref{P:chainmem},
 unique ergodicity  holds for the chain with weighted relocations $\X$. 
The logarithmic function is strictly concave;   we have thus for any $\t\in \S$ that 
$$\log \Kk h(\t) = \log\left( r \sum_{i\geq 0} \tau(i) h(t_i)\right) \geq  \log r + \sum_{i\geq 0} \tau(i)  \log h(t_i). $$
Further, this inequality is strict except if the function $i\mapsto  \log h(t_i)$ is constant on the support of $\tau$. 
We observe by induction that since the Doeblin measure $\boldsymbol{\mu}_h$ of $\X$ is invariant for $\G^*$, all its one-dimensional marginal laws are identical, and for all $i\geq 0$, 
$$\int_{\S}\boldsymbol{\mu}_h(\d \t)   \log h(t_i) =\int_{\S}\boldsymbol{\mu}_h(\d \t)   \log h(t_0). $$
 Hence we have 
 \begin{equation}\label{E:boundr}
 \int_{\S} \boldsymbol{\mu}_h(\d \t) \log\left( \frac{ \Kk h(\t)}{{h}(t_0)}\right) \geq \log r,
 \end{equation}
 and again this inequality is strict except if  the function $i\mapsto  \log h(t_i)$ is constant on the support of $\tau$ for $\boldsymbol{\mu}_h$-almost every $\t\in \S$. 
 
 We know from Lemma \ref{L:tauunbounded}  that the support of the Doeblin measure $\boldsymbol{\mu}_h$ is the whole of $\S$. 
Observe also that the eigenfunction $h$ is not constant on $S$, since otherwise 
the sub-stochastic matrix $\sigma$ would be proportional to a stochastic matrix and this has been excluded from the start.  
Since the support of $\tau$ contains at least two points,  the inequality \eqref{E:boundr} is strict and we conclude the proof with an appeal to Corollary \ref{C2}.

(i) Without loss of generality, we may assume that $\tau$ has unbounded support since otherwise  \eqref{E:tauf1} holds and  the claim is plain from (ii) and (iii).
We fix any $d\geq 1$ sufficiently large so that  the truncated law below is not a Dirac mass, and 
$\overline{\tau}(d)= \sum_{i\leq d} \tau(i)>0$, and consider the law $\tau'$ on $\{0,1, \ldots d\}$ given by $\tau'(j)= \tau(j)/\overline{\tau}(d)$ for $j=0, \ldots, d$. Write also $\sigma'= \overline{\tau}(d)\sigma$. 

 It is immediate to verify  that the killed chain with relocations induced by $\sigma'$ and $\tau'$ can be obtained from the killed chain with relocations induced by  $\sigma$ and $\tau$  by an additional  killing at an independent time with geometric distribution with success parameter $1-\overline{\tau}(d)$. 
 In particular, the lifetime $\z'$ of the former is stochastically dominated by the lifetime $\z$ of the latter, and 
 by (ii), we have therefore 
    $$\liminf_{n\to \infty} \frac{1}{n} \log  \P_{\boldsymbol{\nu}}(\z>n) > \log (\overline{\tau}(d)r),$$
since $\overline{\tau}(d) r$ is the Perron--Frobenius eigenvalue of $\sigma'$. It now suffices to let $d\to \infty$. 
\end{proof} 

\subsection{Large deviation lower bound for empirical measures} \label{sec:LD}

Roughly speaking, the purpose of this section is to point at a striking  reinforcement of Theorem \ref{T2}(i)
when one further focuses on trajectories in $S$ that visit states with pre-described frequencies. Formally, our analysis is done in the framework of the theory of large deviations;  we first introduce some notation and recall a well-known result in this field. 

Recall that for each finite path $\omega=(\omega_0, \omega_1, \ldots, \omega_{\zeta})$  in $S$  with duration $\zeta\geq 1$,
the empirical measures of $\omega$
are defined as 
$$L_n(\omega)= \frac{1}{n} \sum_{i=1}^{n} \delta_{\omega_i} \in \calP(S), \qquad \text{for  }n\leq \zeta.$$
We are interested in these (defective) random variables under two probability distributions on the space of finite paths.
The first is the benchmark  law $P_s$ of the killed Markov chain on $S$ started from $s$ and driven by the sub-stochastic matrix $\sigma$, 
and the second is the law $\P_{\s}$ of the version with relocations, started from some $\s\in \S$. 
Just as in \cite[Exercise 3.1.11]{DZ}, we introduce the rate function $I: \calP(S)\to [0, \infty]$  by the variational formula\footnote{The  reader is invited to observe the similarity with the lower bound appearing in Corollary \ref{C2}.}
$$I(\nu)=
\sup \sum_{s\in S} \nu(s) \log\left( \frac{a(s)}{\sigma a(s)} \right), $$
where  the supremum is taken over all  functions $a\in (0,\infty)^S$.
For any $s\in S$, the sequence of random variables $(L_n)_{n\geq 0}$ satisfies the large deviation principle under $P_s$ with the (convex good) rate function $I$;  see e.g. \cite[Section 3.1]{DZ}. Namely, we have for  every open set 
$G\subset \calP(S)$,
$$\liminf_{n\to \infty} \frac{1}{n} \log P_s\left( L_n\in G, \zeta \geq n\right) \geq - \inf_{\nu\in G} I(\nu),$$
whereas for every closed set $F\subset \calP(S)$,
$$\limsup_{n\to \infty} \frac{1}{n} \log P_s\left( L_n\in F, \zeta \geq n\right) \leq - \inf_{\nu\in F} I(\nu).$$
In particular, if $I$ is continuous on $\overline G$, then 
$$\lim_{n\to \infty} \frac{1}{n} \log P_s\left( L_n\in G, \zeta \geq n\right) = - \inf_{\nu\in G} I(\nu).$$

 The main result of this section is that the same large deviation lower bound remains valid for the killed chain with relocations.
 
 \begin{theorem} \label{T3} For any  $\s\in \S$ and any  open set 
$G\subset \calP(S)$, we have
$$\liminf_{n\to \infty} \frac{1}{n} \log \P_{\s}\left( L_n\in G, \z \geq n\right) \geq - \inf_{\nu\in G} I(\nu).$$
 
 \end{theorem}

\begin{proof} We shall first prove the claim when the law $\tau$ has a finite support, and then extend to the general case by truncations. 

For some $\lambda\in \R_-^S$, set $a(t)=\e^{\lambda(t)}$ for every $t\in S$ and treat the function $a\in (0,1]^S$  as a column vector. 
Consider  the irreducible and aperiodic sub-stochastic matrix $\sigma_{a}$ with entries 
$$\sigma_{a}(s,t)=\sigma(s,t) a(t) =   \sigma(s,t)\e^{\lambda(t)}, \qquad 
\text{for } s,t\in S,$$
and denote by $r_{a}>0$ its Perron--Frobenius eigenvalue. 
Write $\P^{a}$ for the law of the killed chain with relocation for the same law $\tau$, but for the matrix $\sigma_{a}$ instead of $\sigma$. 
 By the Feynman--Kac formula,  for every $\s\in \S$ and $n\geq 1$, there is the identity
 $$\P_{\s}^{a}(\z \geq n) = \E_{\s}\left( \exp(nL_n \lambda), \z\geq n\right),$$
with the notation for matrix multiplication  $L_n \lambda= \sum_{s\in S} \lambda(s)L_n(s)$.

Suppose now that $\tau$ has a finite support,  so $\P^{a}$ can be viewed as law of some finite order Markov chain with killing, hence as some  irreducible and aperiodic killed Markov chain on a finite state space. If we write $\mathbf{r}_{a}>0$ for its Perron--Frobenius eigenvalue, then we thus have
for any $\s\in \S$,
$$\lim_{n\to \infty} n^{-1} \log \P_{\s}^{a}(\z \geq n) = \log \mathbf{r}_{a}, $$
and therefore also
$$\lim_{n\to \infty} n^{-1} \log  \E_{\s}\left( \exp(n L_n \lambda), \z\geq n\right)= \log \mathbf{r}_{a}.$$
By replacing $\lambda$ by $\lambda+c$ for an arbitrary $c>0$, we immediately see that the limit above holds more generally for any $\lambda\in \R^S$ and hence $a\in(0,\infty)^S$.

This enables us to apply the G\"artner--Ellis theorem (see \cite[Section 2.3]{DZ})
and we get that the sequence of random variables $(L_n)_{n\geq 0}$  in the simplex $\calP(S)$ satisfies the large deviation principle under $\P_{\s}$ with the (convex good) rate function 
$$\mathbf{I}(\nu)=\sup\left \{\nu \lambda-\log \mathbf{r}_{a}: \lambda\in \R^S \right \}.$$
Now, since we know from Theorem \ref{T2}(i), or directly from \cite{BerBer}, that $\mathbf{r}_{a} \geq r_{a}$, there is the inequality
$$\mathbf{I}(\nu)\leq \sup\left \{ \nu \lambda -\log r_{a}: \lambda\in \R^S\right \}.$$
On the other hand, it is known from Theorems 3.1.2 and 3.1.6 and Exercise 3.1.11 in \cite{DZ} that the right-hand side above equals $I(\nu)$.
This proves the large deviation lower bound when $\tau$ has finite support.

Finally, we deal with the case when $\tau$ has unbounded support in the same way as in the proof of Theorem \ref{T2}(i). We use the same notation 
and argument as there to get for any  open set 
$G\subset \calP(S)$ and $\s\in \S$ the bound
$$\P_{\s}\left( L_n\in G, \z \geq n\right) \geq  \P'_{\s}\left( L_n\in G, \z \geq n\right),$$
where $\P'$ designates the law of the killed chain with relocations for the matrix $\sigma'=\overline \tau(d) \sigma$ and the law $\tau'$.
Since $\tau'$ has finite support, we deduce from above that
$$\liminf_{n\to \infty} \frac{1}{n} \log \P_{\s}\left( L_n\in G, \z \geq n\right) \geq - \inf_{\nu\in G} I'(\nu),$$
where $I'$ is the rate function for the sub-stochastic matrix $\sigma'$. It is immediately checked  $I'=I-\log \overline \tau(d)$ and we conclude the proof by letting $d\to \infty$.
\end{proof}

\section{Widely dispersed relocations} \label{sec:disperse}

We turn our attention to the analysis of persistence for a widely spread out relocation law $\tau$. Informally speaking,  a dispersed relocation should induce averaging effects, and in particular we may expect that  the occupation measure $\q(\s)$ of a typical $\s\in S$ in the sense of \eqref{E:occmeas}, should concentrate around some $p\in \calP(S)$.
Thanks to \eqref{E:T1alt}, this should make the bound in Corollary \ref{C2} easier to evaluate. 
We shall see that this intuition can be made rigorous, and consider a family of laws on $\N$, $\tau_\eps$ with $ \eps>0$ such that
\begin{equation}\label{E:spread}
       \max_{i\geq 0} \tau_\eps(i)\leq \eps.
\end{equation}
Therefore we now incorporate $\eps$ as an index --or sometimes rather as an exponent or as a variable for convenience-- in our previous notation to stress dependency on this parameter. 

We further assume that the entries of $\sigma$ are positive and that \eqref{E:tauf2} holds for each $\eps>0$, 
so Proposition \ref{P:chainmem}
ensures unique ergodicity of the chain with kernel $\g_{\eps}$ defined  in \eqref{E:boldpi} for some arbitrary  function $a\in (0,\infty)^S$. We write
$\boldsymbol{\mu}_{\eps,a}$ for its Doeblin measure.
The key observation here is that the pushforward image of $\boldsymbol{\mu}_{\eps,a}$ by the occupation measure kernel $\q_{\eps}$ defined in \eqref{E:occmeas}
has a remarkably simple limit as $\eps\to 0+$.

 We have to introduce first some more notation.
Recall that  $\sigma_a$ is the square matrix with entries 
$ \sigma_a(s,t)=\sigma(s,t)a(t)$
and define the nonlinear map $\Phi_a: \calP(S)\to \calP(S)$  by 
\[
               \Phi_a(p)
        =
        \frac{p\sigma_a}{p\sigma_a 1_S},
        \qquad\text{for }
        p\in \calP(S),
\]
where $p\sigma_a$ is a row vector, and 
$1_S$ stands for the  column vector with all coordinates equal to $1$. 
With this notation at hand, observe that 
\[
        \Kk_{\eps} \bar a(\s)
        =
        \q_{\eps}(\s)\sigma a
        = \q_{\eps}(\s)\sigma_a 1_S ,
\]
so 
 \eqref{E:boldpi} becomes
\begin{equation}
\label{eq: gPhi}
        \g_{\eps}(\cdot \s)=\Phi_a\circ \q_{\eps}(\s), \qquad \text{for }\s\in \S.
\end{equation}
We also recall that $r_a>0$ is the Perron--Frobenius eigenvalue of $\sigma_a$   and denote its
normalized row eigenvector  by $\uprho_a\in\calP(S)$, that is, 
\begin{equation}\label{eq: fixed point}
        \uprho_a\sigma_a=r_a \uprho_a       ,\qquad \text{ or equivalently,  } \qquad  \Phi_a(\uprho_a)=\uprho_a \ \text{and}\  \uprho_a \sigma_a 1_S=r_a.
\end{equation}

\begin{theorem}\label{T4}
Assume that the entries of $\sigma$ are positive and that  the relocation 
laws $\tau_\eps$  satisfy \eqref{E:tauf2} for all $\eps>0$ and \eqref{E:spread}.
Then for every $a\in (0,\infty)^S$,  the pushforward measures
$\boldsymbol{\mu}_{\eps,a}\circ\q_\eps^{-1}
$
on $\calP(S)$  of the Doeblin measure $\boldsymbol{\mu}_{\eps,a}$ of $\g_{\eps,a}$ converge weakly, as $\eps\to0+$, to the Dirac mass at
$\uprho_a$. Moreover, the one-dimensional marginals of
$\boldsymbol{\mu}_{\eps,a}$ converge weakly to $\uprho_a$.
\end{theorem}

\begin{proof}
Let $\X^\eps=(\X^\eps(n))_{n\in\Z}$ be the stationary  two-sided version of the
chain with complete connections, with transition kernel $\g_{\eps,a}$ and
Doeblin measure $\boldsymbol{\mu}_{\eps,a}$.  We use a two-sided stationary
version to simplify notation for the past.
Consider also the occupation measure
process in $\calP(S)$,  $ \theta_\eps= \left(  \theta_\eps(n)\right)_{n\in \Z}$, where
\[
        \theta_\eps(n)=\q_{\eps}(\X^\eps(n)).
\]
By compactness of the space $\calP(S)$, we can henceforth assume that $\eps$ goes to $0$ along some sequence for which  $ \theta_\eps(0)$ 
converges in distribution to some random probability measure on $S$, say $\vartheta$.
 We write $K\subset \calP(S)$ for the topological support of the law of $\vartheta$ and will argue that $K$ must reduce to  a singleton in $\calP(S)$ and hence $\vartheta$ is deterministic.

If we write $\Delta_{n}^\eps\in \calP(S)$ for the Dirac mass at $X_0^\eps(n)$, and recalling that by \eqref{eq: gPhi},
the conditional expectation of $\Delta_{n+1}^\eps$ given $\X^\eps(n)$ is $
        \Phi_a(\theta_\eps(n))$, 
we have the representation
\[
        \theta_\eps(n)
        =
        \sum_{i=0}^{\infty}\tau_\eps(i)\Delta_{n-i}^\eps
        =
          \sum_{i=0}^{\infty}\tau_\eps(i)\Phi_a(\theta_\eps(n-i-1))
          +\sum_{i=0}^{\infty}\tau_\eps(i)\left(\Delta_{n-i}^\eps-\Phi_a(\theta_\eps(n-i-1))\right).
\]
The idea is the following: the second sum in the right-hand side  is a  sum of a weighted martingale difference (in $\calP(S)\subset \R^S$), so under \eqref{E:spread}, it is negligible in $L^2$ when $\eps$ is small. For the first sum  in the right-hand side,
we use stationarity: since all the random probability measures  $\theta_\eps(j)$ have the same
 law as $\theta_\eps(0)$, the topological  support $K$ of the limiting random probability measure $\vartheta$ should satisfy
\begin{equation}
    \label{eq: inclusion}
        K\subseteq \operatorname{conv}\Phi_a(K),
\end{equation}
with the notation $\operatorname{conv} A$ for the closed convex hull of $A\subset \calP(S)$.
However $\Phi_a(\cdot)$ acts as a normalized action of the strictly positive $\sigma_a$.  It is well-known that Hilbert's projective metric $\mathbf{d}_H$ is particularly convenient for analyzing the latter:
multiplication by a strictly positive matrix is a strict contraction in $\mathbf{d}_H$. Since normalization does not change the projective distance, 
$\Phi_a$ is also a strict contraction in it.  Therefore \eqref{eq: inclusion} forces $K$ to be a singleton, and the singleton
must be  the Perron fixed point $\uprho_a$.

We now make this argument precise, starting with \eqref{eq: inclusion}.
Fix $n=0$, and write
\[
        \overline\theta_\eps
        =
        \sum_{i=0}^{\infty}\tau_\eps(i)
        \Phi_a(\theta_\eps(-i-1))
        \quad \text{ and }\quad
        R_\eps
        =
        \sum_{i=0}^{\infty}\tau_\eps(i)M_{-i}^\eps, \qquad
        \text{ so that }
        \theta_\eps(0)
        =
        \overline\theta_\eps+R_\eps .
\]
Since the martingale differences $(M^\eps_i)$ are bounded and orthogonal in $L^2$,
\[
        \E\left[|R_\eps|^2\right]
        \leq
        C\sum_{i=0}^{\infty}\tau_\eps(i)^2
        \to0,
\qquad 
\text{ as }
        \eps\to 0+.
\]
Recall that  $\vartheta$ is a random variable in $\calP(S)$ distributed as a weak limit of 
$\theta_\eps(0)$ and that $K\subset \calP(S)$ is the topological support of the law of $\vartheta$.
Fix $\eta>0$, and let $K^\eta$ be the $\eta$-neighborhood of
$K$ in $\calP(S)$, say for the total variation distance. So $
        \P(\vartheta \notin K^\eta)=0$, and by Portmanteau theorem, $\P(\theta_\eps(0)\notin K^\eta)\to0$ as $\eps\to 0+$.
Therefore, using the stationarity of $(\theta_\eps(i))$, we get
\[
        \E\left[
        \sum_{i=0}^{\infty}\tau_\eps(i)
        \mathbf 1_{\{\theta_\eps(-i-1)\notin K^\eta\}}
        \right]
        =
        \sum_{i=0}^{\infty}\tau_\eps(i)
        \P(\theta_\eps(-i-1)\notin K^\eta)
        =
        \P(\theta_\eps(0)\notin K^\eta)
        \to0.\]
Hence
\[
        \sum_{i=0}^{\infty}\tau_\eps(i)
        \mathbf 1_{\{\theta_\eps(-i-1)\notin K^\eta\}}
        \to0
        \qquad\text{in } L^1,
\]
and \textit{a fortiori}
\[
        \operatorname{dist}\left(
        \overline\theta_\eps,
        \operatorname{conv}\Phi_a(K^\eta)
        \right)
        \to0
        \qquad\text{in }L^1,
\]
with the usual $\operatorname{dist}(x,A) = \inf\{ |x-y| \text{ for } y \in A \} $.
Since  $\theta_\eps(0)-\overline\theta_\eps
        =
        R_\eps
        \to0$ in $L^2$, the triangle inequality gives also
\[
        \operatorname{dist}\left(
        \theta_\eps(0),
        \operatorname{conv}\Phi_a(K^\eta)
        \right)
        \to0
        \qquad\text{in }L^1.\]
The map 
 $p\mapsto         \operatorname{dist}\left(
        p,
        \operatorname{conv}\Phi_a(K^\eta)
        \right)$
 being continuous  on $\calP(S)$, the weak convergence of $\theta_\eps(0)$ to $\vartheta$ implies 
 $\E\left[\operatorname{dist}(\vartheta , \operatorname{conv}\Phi_a(K^\eta))\right] = 0$,   
 that is,
 $$\vartheta \in \operatorname{conv}\Phi_a(K^\eta) \qquad \text{a.s.}$$
 We thus have $K\subset \operatorname{conv}\Phi_a(K^\eta)$, and  \eqref{eq: inclusion} follows by letting $\eta\to 0+$.

We continue with presenting the Hilbert projective distance $\mathbf{d}_H$. For any pair of row vectors $x,y \in (0,\infty)^S$ (in other words, $x$ and $y$ are measures on $S$ with full support), define the pseudo-distance 
\[
\mathbf{d}_H(x,y) = \log \frac{\max_s x(s)/y(s)}{\min_t x(t)/y(t)} = \max_{s,t}\log\frac{x(s)y(t)}{x(t)y(s)}.
\]
This can be thought of as a distance between the rays in $(0,\infty)^S$ starting from the origin and passing through $x$ and $y$. Its relevance stems from the strict contracting property, see \cite[Lemma 1]{Birkhoff}. Since the matrix $\sigma_a$ has positive entries, there exists $\kappa_a \in (0,1)$ such that 
\begin{equation} \label{E:Hcontract}
\mathbf{d}_H(\Phi_a(x),\Phi_a(y))=
\mathbf{d}_H(x\sigma_a,y\sigma_a)
\leq \kappa_a \mathbf{d}_H(x,y).
\end{equation}

Next define the Hilbert-diameter of a set $A \subseteq \calP(S)$ by $    \operatorname{diam}_H(A)
        =
        \sup_{x,y\in A}\mathbf d_H(x,y)$ 
        and observe that there is the identity
 \begin{equation} \label{E:disconv}
    \operatorname{diam}_H(\operatorname{conv}A)
        =
        \operatorname{diam}_H(A).
        \end{equation}
Indeed, we have by definition that for every $u,v\in A$ and every  $s,t\in S$,
\[
       u(s)v(t)\leq \e^{\operatorname{diam}_H(A)} u(t)v(s).
\]
Then consider two sequences  in $A$, $(u^k)_{k\geq 0}$ and $(v^\ell)_{\ell\geq 0}$, and two probability measures  on $\N$, $(\alpha_k)_{k\geq 0}$ and  $(\beta_\ell)_{\ell\geq 0}$,
and  set
\[
        x=\sum_k\alpha_k u^k
        \qquad \text{and}\qquad
        y=\sum_\ell\beta_\ell v^\ell.
\]
We have for any $s,t\in S$
\[
        x(s) y(t)
        =
        \sum_{k,\ell}\alpha_k\beta_\ell u^k(s) v^\ell(t)\leq
       \e^{\operatorname{diam}_H(A)}
       \sum_{k,\ell}\alpha_k\beta_\ell u^k(t) v^\ell(s) 
       =
       \e^{\operatorname{diam}_H(A)} x(t) y(s).
        \]
This shows $\mathbf d_H(x,y)\leq {\operatorname{diam}_H(A)}$, and thus establishes \eqref{E:disconv}.

We now apply this to the limiting support $K$. From the previous step,
\[
        K\subseteq\operatorname{conv}\Phi_a(K).
\]
Since $\sigma_a$ is strictly positive, $\Phi_a(\calP(S))$ is contained in the
interior of $\calP(S)$, and therefore $K$ has finite Hilbert diameter. Using \eqref{eq: inclusion}, \eqref{E:Hcontract}, and \eqref{E:disconv}, we get
\[
\operatorname{diam}_H(K)
        \leq
        \operatorname{diam}_H(\operatorname{conv}\Phi_a(K))
        \leq
        \operatorname{diam}_H(\Phi_a(K))
        \leq
        \kappa_a\operatorname{diam}_H(K),\]
and because $\kappa_a<1$, this implies that $\operatorname{diam}_H(K)
=0$. Hence all points of $K$ lie on one ray, but  they are all restricted to the simplex $\calP(S)$,  they in fact should coincide, that is, $K=\{p\}$ for some $p$.

Using again \eqref{eq: inclusion}, we get $p=\Phi_a(p)$. However, Perron theory ensures uniqueness of the fixed point $\uprho_a$ of $\Phi_a$ and 
$
        K=\{\uprho_a\}.
$
Thus any weak limit  of $\boldsymbol{\mu}_{\eps,a}\circ \q_\eps^{-1}$ is $\delta_{\uprho_a}$, and consequently
\[
  \lim_{\eps\to 0+}      \boldsymbol{\mu}_{\eps,a}\circ \q_\eps^{-1}
      =         \delta_{\uprho_a},
        \qquad \text{weakly}.
\]
Finally, the one-dimensional marginal laws of
$\boldsymbol{\mu}_{\eps,a}$ are all identical, and by \eqref{eq: gPhi}, they are given by 
\[
             \int_{\S}
        \boldsymbol{\mu}_{\eps,a}(\d\s)
        \Phi_a(\q_\eps (\s))=
            \int_{\calP(S)}
        \left(\boldsymbol{\mu}_{\eps,a}\circ \q_\eps^{-1}\right)(\d p)
        \Phi_a(p).
\]
Since $\Phi_a$ is continuous, this marginal converges to
$\Phi_a(\uprho_a)=\uprho_a$.
\end{proof}

The figure below illustrates  the concentration stated in Theorem  \ref{T4}. Numerical simulations have been performed for a state space $S$ with 2 elements and 
the sub-stochastic matrix $\sigma=\left( ^{0.72 \ 0.08}_{0.18\ 0.58} \right) $
and geometric relocation laws $\tau_\eps(k)=\eps(1-\eps)^k$ for $k\geq0$. The space $\calP(S)$ of probability measures on $S$ reduces to the  family 
of Bernoulli distribution with some parameter in $[0,1]$, and  
the occupation measure $ \theta_\eps(0)$ is thus represented by a random variable in $[0,1]$.
\begin{figure}[H]
\centering
\includegraphics[width=0.78\textwidth]{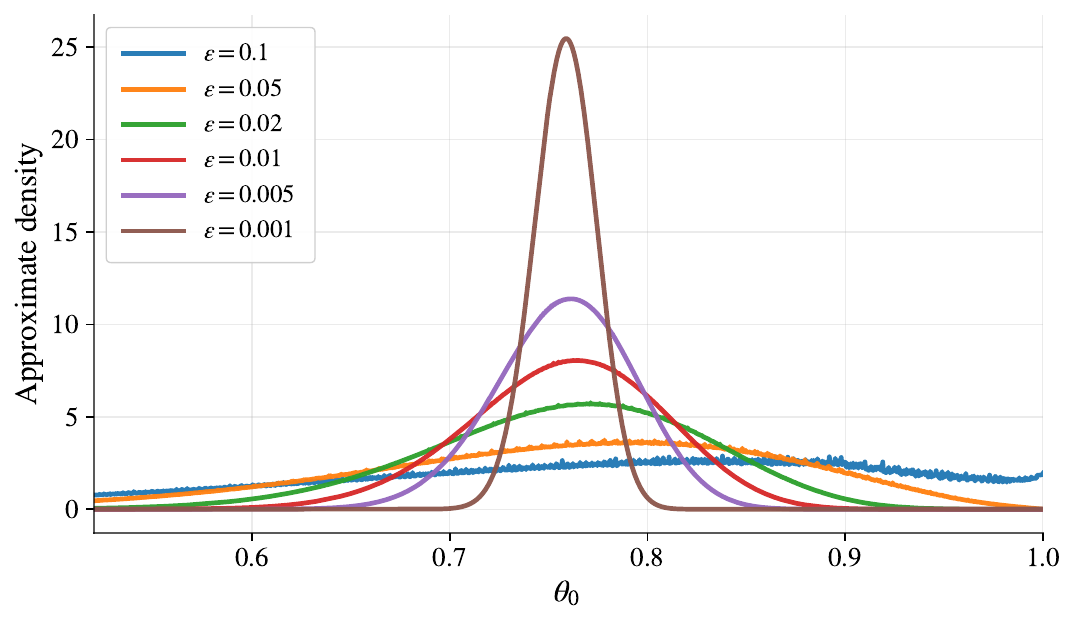}
\caption{Concentration of the occupation measure $\theta_\eps(0)$ for geometric relocation laws $\tau_\eps(k)=\eps(1-\eps)^k$, for six decreasing values of $\eps$.}
\label{fig:concentration}
\end{figure}

Theorem \ref{T4}  immediately yields a quantitative lower bound for
the spectral radius  $\mathbf{r}(\eps)$  of the killed chain with complete connections with a widely dispersed relocation law $\tau= \tau_\eps$.
Recall from  Gelfand's formula that  $\mathbf{r}(\eps)$ is related to persistence  by 
\[ \mathbf{r}(\eps) = \lim_{n\to\infty} \sup_{\boldsymbol{\nu}\in \calP(\S)}  \P^\eps_{\boldsymbol{\nu}}(\z >n)^{1/n},
\]
where $\P^\eps_{\boldsymbol{\nu}}$ denotes the law of the killed chain started from $\boldsymbol{\nu}$. We  arrive at the following asymptotic lower    bound.

\begin{corollary}\label{C3} Under the assumptions of Theorem \ref{T4}, we have 
\[
        \liminf_{\eps\to0+}
      \mathbf{r}(\eps)
        \geq
        \sup \left\{
        r_a
        \exp(-
       \uprho_a \log a) : a\in(0,\infty)^S
        \right\},
\]
where, for each function ${a\in(0,\infty)^S}$, $r_a$  is the Perron eigenvalue 
of the matrix  $\sigma_a$, and the quantity $ \uprho_a \log a$ denotes the mean value of the function $\log a$ with respect to the probability measure $\uprho_a$ corresponding to the normalized row eigenvector.
\end{corollary}

\begin{remark}
Recalling \eqref{E.specm} and taking $a=h$, we immediately deduce the more explicit lower bound 
\[
        \liminf_{\eps\to0+}
        \mathbf{r}(\eps)
        \geq
         r   \exp(-
       \uprho_h  \log h) \uprho_h h .
     \]
We stress that, by  Jensen's inequality, the right-hand side is larger than the spectral radius $ r$ of $\sigma$; compare with Theorem \ref{T2}(i). 
\end{remark} 

\begin{proof} We know from Corollary \ref{C2} and \eqref{E:T1alt} that for any $a\in(0,\infty)^S$, 
\[ \log  \mathbf{r}(\eps)
\geq    \int_{\calP(S)}
        \left(\boldsymbol{\mu}_{\eps,a}\circ \q_\eps^{-1}\right)(\d p)
        \Big(
        \log(p \sigma a)
        -
       p \log a
        \Big).
\]
Our claim follows from Theorem \ref{T4} since
$\uprho_a \sigma a = \uprho_a\sigma_a 1_S = r_a$.
\end{proof}

 Figure~\ref{fig:log-radius} illustrates Corollary~\ref{C3} for the same example as in Figure~\ref{fig:concentration}, now for twelve small values of $\eps$.

\begin{figure}[H]
\centering
\includegraphics[width=0.78\textwidth]{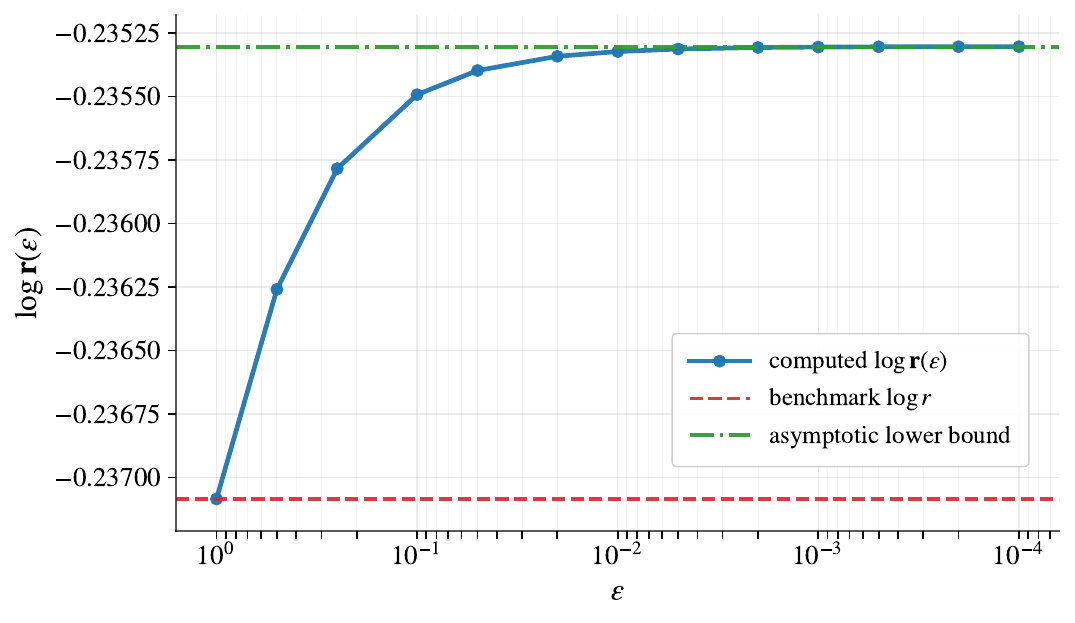}
\caption{Numerical logarithmic comparison of $\mathbf r(\eps)$, the benchmark value $r$, and the asymptotic lower bound from Corollary~\ref{C3}, for geometric relocation laws $\tau_\eps(k)=\eps(1-\eps)^k$.}
\label{fig:log-radius}
\end{figure}

It is remarkable that the lower bound for the limit inferior of $\mathbf r(\varepsilon)$ given there appears to be the exact value of the limit. Moreover, $\mathbf r(\varepsilon)$ seems always to lie between the lower bound $r$ and this limit. We do not have a satisfactory explanation for these observations, and we therefore raise the following question: for which substochastic matrices $\sigma$ does the upper bound
$$
\mathbf{r}(\tau) \le \sup \left\{ r_a \exp\bigl(-\uprho_a \log a\bigr) : a \in (0,\infty)^S \right\}
$$
hold for all relocation laws $\tau$, say with exponentially decaying tails? Of course, when this is the case, Corollary~\ref{C3} can then be strengthened into
\[
        \lim_{\eps\to0+}
        \mathbf{r}(\eps)
        =           \sup \left\{
        r_a
        \exp(-
       \uprho_a \log a) : a\in(0,\infty)^S
        \right\}.
     \]

\bibliography{Reloc.bib}

\end{document}